\documentclass[11pt,fleqn]{article}

\usepackage{amsmath,amssymb,graphicx,bm}
\usepackage{pgf,tikz}
\usetikzlibrary{arrows}
\usepackage{epstopdf}
\usepackage{blkarray}
\usepackage{colortbl,xcolor}
\usepackage{verbatim}
\begin{document}

\sloppy
\definecolor{cqcqcq}{rgb}{1,1,1}
\newtheorem{axiom}{Axiom}[section]
\newtheorem{claim}[axiom]{Claim}
\newtheorem{conjecture}[axiom]{Conjecture}
\newtheorem{corollary}[axiom]{Corollary}
\newtheorem{definition}[axiom]{Definition}
\newtheorem{example}[axiom]{Example}
\newtheorem{fact}[axiom]{Fact}
\newtheorem{lemma}[axiom]{Lemma}
\newtheorem{observation}[axiom]{Observation}
\newtheorem{proposition}[axiom]{Proposition}
\newtheorem{theorem}[axiom]{Theorem}

\renewcommand{\topfraction}{1.0}
\renewcommand{\bottomfraction}{1.0}

\newcommand{\proof}{\emph{Proof.}\ \ }
\newcommand{\qed}{~~$\Box$}
\newcommand{\rz}{{\mathbb{R}}}
\newcommand{\nz}{{\mathbb{N}}}
\newcommand{\zz}{{\mathbb{Z}}}
\newcommand{\eps}{\varepsilon}
\newcommand{\cei}[1]{\lceil #1\rceil}
\newcommand{\flo}[1]{\left\lfloor #1\right\rfloor}
\newcommand{\seq}[1]{\langle #1\rangle}
\newcommand{\p}{\star}
\newcommand{\mf}[1]{\small{\emph{#1}}}
\newcommand{\mfb}[1]{\small{\color{blue}\bf{\emph{#1}}}}


\title{The bipartite travelling salesman problem: A pyramidally solvable case}
\author{
Vladimir G.\ Deineko\thanks{{\tt Vladimir.Deineko@wbs.ac.uk}.
Warwick Business School, The University of Warwick, Coventry CV4 7AL, United Kingdom}
\and
Bettina Klinz\thanks{{\tt klinz@math.tugraz.at}.
Institute of Discrete Mathematics, TU Graz, Austria}
\and
Gerhard J.\ Woeginger\thanks{{\tt Deceased April 1, 2022}.
  Department of Computer Science, RWTH Aachen, Germany}
}
\date{}
\maketitle

\begin{abstract}
  In the bipartite travelling salesman problem (BTSP), we are given $n=2k$
  cities along with an $n\times n$ distance matrix and a partition of
  the cities into $k$ red and $k$ blue cities. The task is to find a
  shortest tour which
alternately visits blue and red cities. We consider the BTSP restricted to
the class of so-called Van der Veen distance matrices. We show that
this case remains NP-hard in general but becomes solvable in polynomial time
when  all vertices with odd indices are coloured blue and all
with even indices are coloured red. In the latter case an optimal solution
can be found in $O(n^2)$ time among the set of pyramidal tours.

\medskip\noindent{\bf Keywords.}        
Combinatorial optimisation, bipartite travelling salesman problem,  pyramidally solvable case, Van der Veen matrix, recognition algorithm.
 
\end{abstract}

\section{Introduction}
In the \emph{travelling salesman problem} (TSP) a salesperson is looking for
the shortest tour to visit all cities from a given list of cities. The input
consists of $n$ cities (including the city where the
salesperson lives), and the distances (or times) of travelling between each
pair of cities. The task is to find the shortest (cyclic) tour visiting all
the cities. The TSP is probably one of the best studied NP-hard
optimisation problems and has served as important benchmark problem in
discrete optimisation with a long list of outstanding contributions
to the theory and practice of the field (see e.g.\ the
monographs~\cite{ABCC,Gutin,TSP}). One of the well
established directions of research for NP-hard
optimisation problems is the  investigation of polynomially solvable special
cases (see the surveys~\cite{BSurv,DKTW,GLS,Kabadi} for further references). 

In the \emph{bipartite travelling salesman problem} (BTSP) the set of cities
$\{1,\ldots,n\}$ is partitioned into two subsets, the set
$K_1$ of blue cities and the set
$K_2$ of red cities where $|K_1|=|K_2|=k$, $n=2k$. Any feasible tour in the
BTSP has to alternate between blue and red cities. The objective 
is to find the shortest such tour.
\medskip

{\bf Motivation and previous work.}
An important application of the BTSP can be found in the context of container
terminal management
(see~\cite{BierwirthMeisel-a,BierwirthMeisel-b,Boysen,Carlo,LehnfeldKnust}). 
In a container terminal trains with containers arrive to a terminal and have
to be unloaded to a storage area. The containers have fixed positions on the
trains and the unloading is performed by a single crane. The goal is to minimise
the unloading time. The special case with only $k$ storage positions
specified for the locations of $k$ containers from the train, can be modelled
as BTSP.  The BTSP has also drawn the attention of researchers
(\cite{Baltz,BaltzSri,Chalasani,Frank}) due to its relevance
to pick-and-place (or grasp-and-delivery) robots
(\cite{Anily,Atallah,Lei,Michel,Karuno}).

The BTSP is not as well studied as the TSP. In particular, while there are
plenty of publications on polynomially solvable cases of the TSP, we are
aware of only a few papers~\cite{DW2014,Garcia,Halton,Mis} published on
solvable cases of the BTSP.

Halton~\cite{Halton} was the first who provided a polynomially solvable case
of the BTSP. He considered the shoelace problem where
cities represent the eyelets of shoes and the objective is to find
an optimal shoe lacing strategy that minimises the length of the 
shoelace. In Halton's model the eyelets can be viewed as points in
the Euclidean plane: the blue points $K_1=\{1,2,\ldots,k\}$ have coordinates
$(0,d),(0,2d),\ldots,(0,kd)$ and the red points $K_2=\{k+1,k+2,\ldots,n\}$
have coordinates $(a,d),(a,2d),\ldots,(a,kd)$, respectively.
Halton proved that the optimal BTSP solution in this case has a
special structure which is illustrated in Figure~\ref{fig:1H}(a).

\begin{figure}
\unitlength=1cm
\begin{center}
\begin{picture}(15.,5)
{
\begin{picture}(15.,5)
\put(1,0.7){\framebox[0.7\width]{
\includegraphics[scale=1]{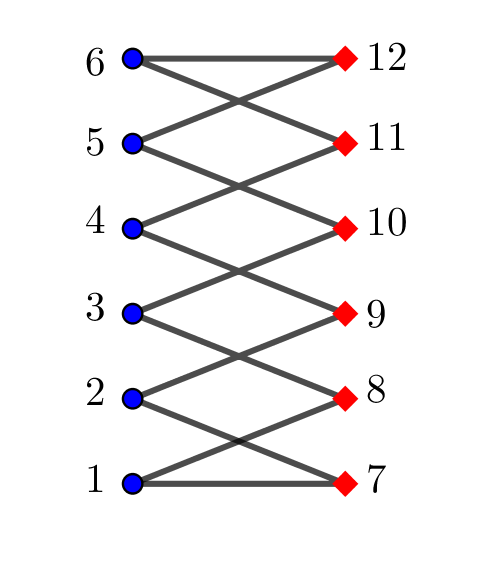}
}}
\put(4.6,0.7){\framebox[0.9\width]{
\includegraphics[scale=1.]{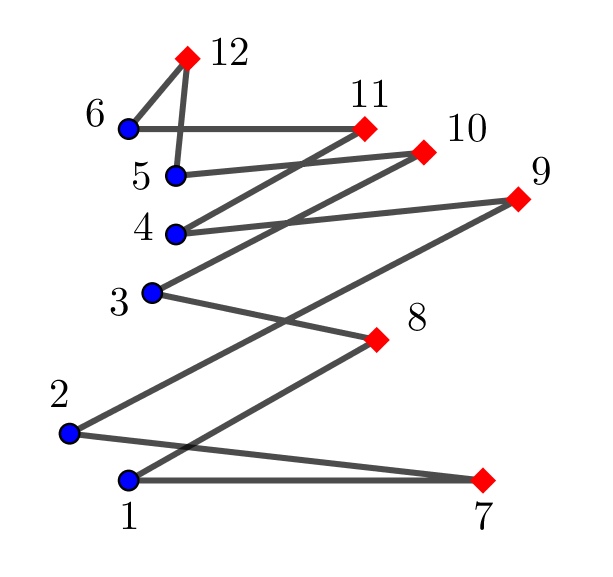}
}}
\put(10,0.7){\framebox[0.7\width]{
\includegraphics[scale=1.]{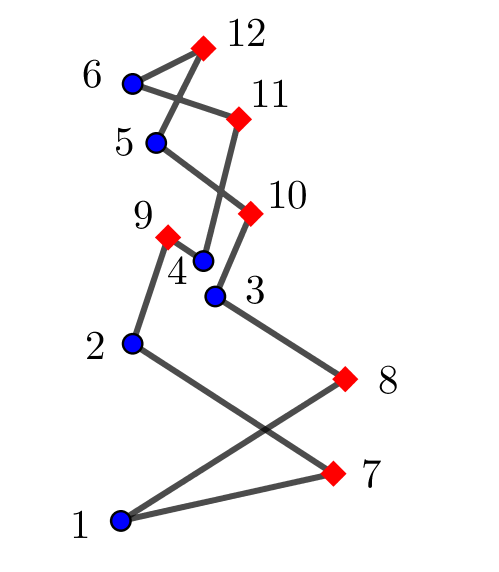}
}}
\put(2,0){(a)}
\put(7,0){(b)}
\put(11,0){(c)}
\end{picture} 
}
\end{picture}
\end{center}

\caption{Illustration from \cite{DW2014} for polynomially solvable BTSP cases as variants of the shoelace problem:
(a) - Halton \cite{Halton} case; (b) - Misiurewicz \cite{Mis} case; (c) case of  Deineko \& Woeginger \cite{DW2014}. 
}  
\label{fig:1H}
\end{figure}
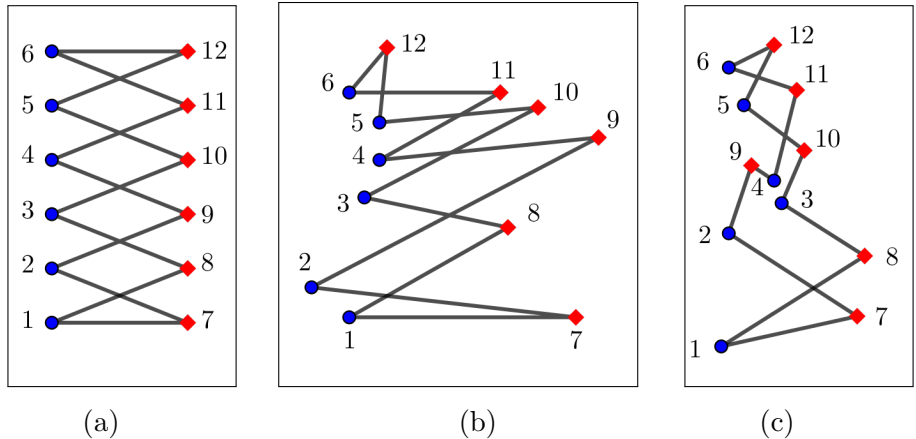

The shoelace problem is a nice interpretation of the BTSP that can be used for
entertaining and educational purposes. 
Misiurewicz~\cite{Mis} argued that Halton's case is based on quite
restricted assumptions which are hardly met in real life. He generalised
Halton's model to the case referred to as
``for old shoes'' (Fig.~\ref{fig:1H}(b)). Deineko and Woeginger~\cite{DW2014}
went on further and investigated the case ``for
\emph{ very old} shoes'' (Fig.~\ref{fig:1H}(c)).  

Notice that the cities in Halton's case are points in the Euclidean plane
and are placed on the boundary of their convex hull. The blue and the
red points occur consecutively along this boundary.
Garcia and Tejel~\cite{Garcia} considered a more general case of the
Euclidean BTSP where the points are still on the boundary of their convex hull,
but the points with the same colour are  not necessarily on consecutive
positions. For this case, they described a candidate set of exponential
size within which the optimal BTSP tour can be found. The running time of
their algorithm is $O(n^6)$.

Recently Bandyapadhyay et al.~\cite{BPath} studied the bipartite TSP-path
problem with the points placed on a line. They described several cases
when the problem can be solved even in linear time.

The special BTSP cases considered in~\cite{BPath, Garcia, Halton} have been
characterised in terms of special locations of points in the Euclidean plane,
while in~\cite{DW2014} and~\cite{Mis} the considered special cases are
obtained by imposing conditions on the entries in the distance matrices.
These conditions are
defined by sets of linear inequalities. This approach is widely used in the
research literature on polynomially solvable cases of the TSP
(see e.g.~\cite{DKTW} and the references therein).
\medskip

{\bf Contribution and organisation of the paper.}
In this paper we investigate the BTSP on Van der Veen distance matrices (see
conditions (\ref{vdv.c}) in Section~\ref{sec:definitions} for a definition)
and the newly introduced class of relaxed Van der Veen matrices
which result if certain of the linear inequalities which are involved
in the definition of Van der Veen matrices are dropped
(for a definition see conditions~(\ref{eq:delta1}) in
Section~\ref{sec:pyramidal}).

The class of Van der Veen matrices 
has been well investigated in the literature on polynomially solvable
cases of the TSP, but
has not been considered in the context of the BTSP. We first show that
the BTSP when restricted to Van der Veen distance matrices remains NP-hard.
Then we show that the even-odd BTSP which results if $K_1$ contains all
cities with odd indices and $K_2$ contains all cities with even indices becomes
solvable in polynomial time  when restricted to (relaxed) Van der Veen
distance matrices. In this  case an optimal
tour can be found within the set of pyramidal tours in $O(n^2)$ time.

We can go one step further. We can recognise the class of matrices $C$
which become relaxed Van der Veen matrices after renumbering the 
cities in $K_1$ and in $K_2$ with independent permutations which allows us to
find an optimal tour for this subclass of permuted relaxed Van der Veen
matrices in $O(n^4)$ time (the time needed by the recognition algorithm).


In Section~\ref{sec:definitions} we provide the definitions and preliminaries
needed in the rest of the paper. Section~\ref{sec:pyramidal} constitutes the heart of
the paper and contains both the hardness result
for the BTSP restricted to general Van der Veen matrices as well as our
polynomial time algorithm for the even-odd BTSP restricted to relaxed
Van der Veen matrices. Section~\ref{sec:RecognitionP} describes a polynomial
time recognition algorithm for a subclass of permuted relaxed Van der Veen
matrices. Section~\ref{sec:conclusion} closes the paper with concluding
remarks.



\section{Definitions and preliminaries}\nopagebreak 
\label{sec:definitions}

Given an $n\times n$ distance matrix $C=(c_{ij})$ the objective in the TSP is to
find a cyclic
 permutation $\tau$ of the set $\{1,2,\ldots,n\}$ that minimises the
 travelled distance $c(\tau)=\sum_{i=1}^{n}c_{i\tau(i)}$.  
 Throughout this paper we assume that distance matrices are \emph{symmetric} matrices.

The cyclic permutations are also called \emph{tours}, the elements 
of  set $\{1,2,\ldots,n\}$ are  called \emph{cities} or \emph{points\/}, and
$c(\tau)$ is referred as the length of the permutation $\tau$.
 The set of all permutations over set $\{1,2,\ldots,n\}$ is denoted by
 ${\mathcal S}_n$. For $\tau\in {\mathcal S}_n $, we denote by $\tau^{-1}$ the
\emph{inversion} of $\tau$, i.e., the permutation for which
 $\tau^{-1}(i)$ is the predecessor of $i$ in the tour $\tau$, for
 $i=1,\ldots,n$. We also use a cyclic representation of a cyclic
 permutation $\tau$ in the form
\begin{eqnarray*}
\tau=\seq{i,\tau(i),\tau(\tau(i)),\ldots,\tau^{-1}(\tau^{-1}(i)),\tau^{-1}(i),i}.
\end{eqnarray*}

In the bipartite TSP (BTSP) on top of an $n\times n$ distance matrix $C$
we are also given a partition of the $n=2k$ cities into the two sets $K_1$ and
$K_2$ with  $K_1\cup K_2=\{1,2,\ldots,n\}$ and $|K_1|=|K_2|=k$.
The special case of the {\em even-odd BTSP}  results when 
$K_1$ contains all cities with odd indices and $K_2$ contains all cities
with even indices.

The set
${\mathcal T}_n(K_1,K_2)$ of all feasible tours for the BTSP 
can formally be defined as
\begin{eqnarray*}
{\mathcal T}_n(K_1,K_2)=\{\tau\in {\mathcal
S}_n|\tau^{-1}(i),\tau(i)\in K_2 \textrm{ for }
 i\in K_1;  
\tau^{-1}(i),\tau(i)\in K_1 \textrm { for } i\in K_2\}.
\end{eqnarray*}

We will refer to the tours in ${\mathcal T}_n(K_1,K_2)$ as
\emph{bipartite} tours or feasible BTSP tours.

For example, if $K_1:=\{1,2,\ldots,k\}$ and $K_2:=\{k+1,\ldots,n\}$, then the tour
\begin{eqnarray*}
 \tau^*=\seq{1,k+1,2,k+3,4,k+5,6\ldots,7,k+6,5,k+4,3,k+2,1}
\end{eqnarray*} 
which is illustrated in Figure~\ref{fig:1H} is a feasible BTSP tour, i.e., is
a member of ${\mathcal T}_n(K_1,K_2)$.

Let $C[K_1,K_2]$ denote the $k\times k$ matrix which is obtained
from  matrix $C$ by \emph{deleting} all rows with indices from $K_2$ and
all columns with indices from $K_1$. Clearly, the length $c(\tau)$ of
any feasible BTSP tour $\tau$ is calculated by using \emph{only}
entries from $C[K_1,K_2]$.

A tour $\tau=\seq{1,\tau_2,\ldots,\tau_m,n,\tau_{m+2},\ldots,\tau_{n-2},1}$ is
called a {\em pyramidal tour}, if $1<\tau_2<\ldots<\tau_m<n$ and 
$n>\tau_{m+2}>\ldots >\tau_{n-2}>1$.  An instance of the TSP/BTSP is called
\emph{pyramidally solvable} if an optimal solution to the instance can be
found within the set of pyramidal tours.
 
 The notion of pyramidal tours is well known in the rich literature on
 polynomially solvable cases of the TSP (see the
 surveys~\cite{BSurv,GLS,Gutin,Kabadi,DKTW}). Although the set of pyramidal
 tours contains $\Theta(2^n)$ tours, a shortest pyramidal can be found in
 $O(n^2)$ time by dynamic programming (see e.g.\ Section 7 in \cite{GLS}). 
 
 A symmetric $n\times n$ matrix $C$ is
 called a \emph{Van der Veen\/} matrix if it fulfils the so-called
 \emph{Van der Veen conditions}
\begin{eqnarray}
c_{ij}+c_{j+1,m}\le c_{im}+c_{j+1,j} &&
\mbox{~for all~} 1\le i<j<j+1<m\le n. \label{vdv.c}
\end{eqnarray}

\section{Complexity results for the BTSP on (relaxed) Van der Veen matrices}\label{sec:pyramidal}

In this section we investigate the complexity of the BTSP restricted to
Van der Veen matrices and relaxed Van der Veen matrices. It was proved by Van der Veen \cite{Veen} back in 1994 that the TSP with a
distance matrix that satisfies conditions (\ref{vdv.c}) is pyramidally
solvable. For the BTSP the situation is more complex.
It turned out that the partitioning into blue and red cities
which is part of the input plays a crucial role.

We will show that the BTSP restricted to Van der Veen matrices remains NP-hard
while the even-odd BTSP where the colouring of the cities is chosen according to
the parity of the city indices becomes polynomially solvable for Van der Veen
matrices and even for relaxed Van der Veen matrices.

We start with the hardness result for the special case where the first
$k$ vertices are coloured blue and the remaining ones are coloured red.

\begin{theorem}
  The BTSP is NP-hard on $n\times n$ Van der Veen distance matrices
  when the $n=2k$ cities are partitioned into the sets
  $K_1:=\{1,2,\ldots,k\}$ and $K_2:=\{k+1,\ldots,2k\}$.
\end{theorem}
\proof The proof follows along the lines of the proofs of similar results 
for the TSP (see \cite{MaxDemi} and \cite{Steiner}) and makes some
adjustments required for the BTSP. The proof is done by a reduction from the
NP-hard HAMILTONIAN CYCLE PROBLEM IN BIPARTITE GRAPHS (cf. \cite{Garey}). 

Let $G=(K_1\cup K_2,E)$  be a bipartite graph with $E\subset K_1\times K_2$.
From $G$ we construct a $2k\times 2k$ Van der Veen matrix $C=(c_{ij})$ as
follows. The items in $C[K_1,K_2]$ are obtained from the adjacency matrix
of $G$: If there is an edge between $i\in K_1$ and $j\in K_2$, we set
$c_{ij}=c_{ji}=0$, otherwise $c_{ij}=c_{ji}=1$.

Notice that any tour for the BTSP involves only trips between cities which are
not on the same side of the partition. The corresponding distances are elements
of the submatrix $C[K_1,K_2]$. Hence
the graph $G$ contains a Hamiltonian cycle if and only 
if the length of an optimal solution to the BTSP instance with matrix $C$
is 0. 

What is left to be shown is that the yet undefined elements in $C$
can be set in a way such that the resulting matrix is a Van der Veen
matrix. This can be achieved as follows.
For $i,j\in K_1=\{1,2,\ldots,k\}$, $i<j$, we set
$c_{ij}=-(k+1)+j$, $c_{ji}=c_{ij}$. We also set $c_{k+i,k+j}=-i$,
$c_{k+j,k+i}=c_{k+i,k+j}$. This completes the construction of the matrix $C$.

We now need to check that the inequalities (\ref{vdv.c}) are fulfilled
to confirm that $C$ is
indeed a Van der Veen matrix. The following five cases need to be considered.
\begin{itemize}
\item[ (a)] $k\in [m,n]$: It follows from $m\le k$ that
  $i<k$, 
  $i<j<j+1<m\le n=2k$. Hence inequalities (\ref{vdv.c}) can be rewritten as
  $-i-(j+1)\le -i -j$, or $-1\le 0$ which is trivially fulfilled.
\item[(b)] $k\in [j+1,m-1]$: It follows from $j+1\le k\le m-1$ that $i\le k$,
  $k<j<j+1<m\le 2k$. Inequalities (\ref{vdv.c}) can be rewritten as $c_{ij}+
  k-j-1\le c_{im}+k-j$, or $c_{ij}-1\le c_{im}$, which is true since $c_{ij}$
  and $c_{im}$ are binary as they are obtained from the adjacency matrix of
  graph $G$.
\item[(c)] $k\in [j,j]$: It means that $i<j=k<k+1<m\le 2k$. In this case  $c_{k+1,k}$ and $c_{im}$ are binary, $c_{ik}= -(k+1)+k=1$, $c_{k+1,m}=-(k+1-k)=-1$. 
The inequality $1-1\le  c_{k+1,k}+c_{im}$ is always satisfied.
\item[(d)] $k\in [i,j-1]$: Similar to case (b).
\item[(e)] $k\in [1,i-1]$: Similar to case (a).
\hfill\qed
\end{itemize}
\medskip

The situation becomes more favourable for the even-odd BTSP.
Remember that there we consider $K_1:=\{1,3,\ldots,2k-1\}$ and
$K_2:=\{2,4,\ldots,2k\}$ with $n=2k$.

\begin{theorem}\label{theo:pyr}
  The even-odd BTSP on $n\times n$ Van der Veen distance matrices
  is pyramidally solvable and hence can be solved in $O(n^2)$ time.
  \end{theorem}
\proof
In this proof the well-known tour improvement technique (cf.\ \cite{BSurv}) is
used. Assume that we are given a bipartite feasible tour on $n=2k$ cities,
with the $k$ blue cities placed on odd positions, and the $k$ red cities
placed on even positions.
We will show how to transform a feasible tour $\tau$ into a pyramidal tour
which is also a \emph{feasible} BTSP solution, i.e., with
the blue cities placed on odd positions, and the red cities
placed on even positions. 

Index $i$ in tour $\tau$ is called a \emph{valley}, if $\tau^{-1}(i)>i$ and
$i<\tau(i)$. Observe that a tour is pyramidal if and only if city $1$
is its only valley.
If tour $\tau$ is not a pyramidal tour, we identify the minimal valley which
 is greater than $1$. Let this valley be $j+1$ and let $\tau(j+1)=m$.
 We further on assume that $\tau(j)=l$ and $l>j$ which can be done w.l.o.g.
 since the distance matrix $C$ is symmetric and we can choose to either
 work with $\tau$ or its inversion $\tau^{-1}$.
 Notice that the cities $j$ and $j+1$ have
 different parity. Assume that $j+1$ is even (a red city), then 
 $j$ is odd (a blue city).

\begin{figure}
\unitlength=1cm
\begin{center}
\begin{picture}(14.4,5)
\put(1.5,0)
{
\includegraphics[scale=1.8]{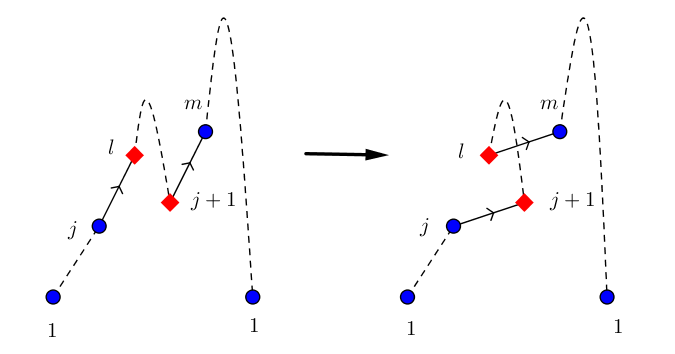}
}
\end{picture}
\end{center}
\caption{ Illustration of one iteration of the tour improvement technique.}
\label{fig:transformPyramidal}
\end{figure}

We now create a new feasible BTSP tour $\tau'$ as follows.
We delete the edges $(j,l)$, $(j+1,m)$, invert the sub-tour
$\seq{l,\ldots,j+1}$ to obtain the sub-tour $\seq{j+1,\ldots,l}$, and
then introduce two new edges $(j,j+1)$, $(l,m)$.  For an illustration see
Figure~\ref{fig:transformPyramidal}. It is obvious that we obtain
this way a new tour which is a feasible solution for the BTSP. Moreover,
the minimal valley in the new tour is bigger than $j+1$. After a finite
number of iterations we end up of a pyramidal tour which is a
feasible solution to the BTSP.

The rest of the proof will be concerned with proving the claim below. Iterative
application implies that the procedure described above transforms
any feasible BTSP tour into a pyramidal BTSP tour with the same
or a smaller total length.
\medskip

\noindent
\emph{ Claim.} $c(\tau')\le c(\tau)$.

\noindent
{\em Proof of the Claim.\/}
Observe that what we need to show is that
\begin{equation}\label{eq:delta}
  c_{j+1,j}+c_{lm} -c_{jl}-c_{j+1,m}\le 0 \qquad\mbox{for all}\ j,l,m:\ 
  j<j+1<l,m\le n
\end{equation}
where we can restrict ourselves to the case where
$(j\equiv m)\ \mathrm{mod}\ 2$ and $j+1\equiv l\ \mathrm{mod}\ 2$.
Let $\Delta(i,l,j,m):= c_{ij}+c_{lm}-c_{im}-c_{jl}$. For a symmetric matrix $C$,
$\Delta(i,l,j,m)=\Delta(j,m,i,l)$. 
Using the $\Delta$-notation, the inequalities (\ref{vdv.c}) can be rewritten as
$\Delta(i,j+1,j,m)\le 0$, or $\Delta(j,m,i,j+1)\le 0$ for all
$1\le i<j<j+1<m\le n$.
The inequalities (\ref{eq:delta}) can be rewritten as
$\Delta(j+1,l,j,m)\le 0$ for all $j,m,l$ with $j+1<l,m\le n$.
It hence suffices to prove the following two statements.
\begin{itemize}
  \item Statement I: $\Delta(j+1,l,j,m)\le 0 \mbox{~~for all~}\ j,m,l \ \mbox{~with~}\ j+1<
    \bm{l<m}\le n.$
  \item Statement II: $\Delta(j+1,l,j,m)\le 0 \mbox{~~for all~}\ j,m,l \ \mbox{~with~}\ j+1<
    \bm{m<l}\le n.$
  \end{itemize}
  
\begin{figure}
\unitlength=1cm
\begin{center}
\begin{picture}(14.5,6)
\put(1.5,-0.5)
{
\includegraphics[scale=1.2]{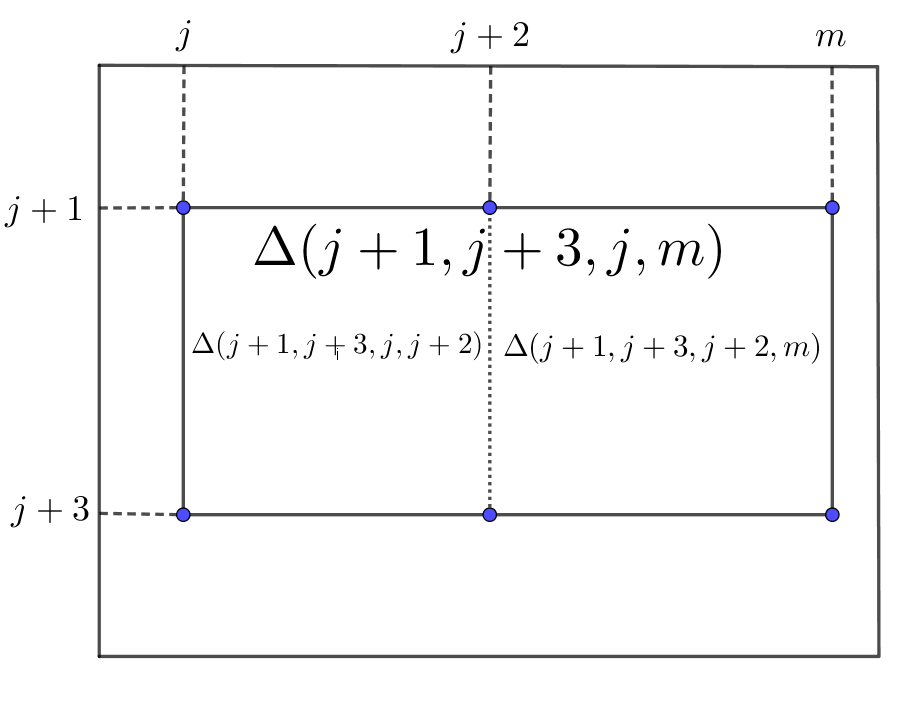}
}
\end{picture}
\end{center}
\caption{ Schematic representation of 
  $\Delta$: $\Delta(j+1,j+3,j,m)=\Delta(j+1,j+3,j,j+2)+\Delta(j+1,j+3,j+2,m)$}
\label{fig:schema}
\end{figure}

\noindent
{\em Proof of Statement I:\/} If $l=j+3$, then it can be easily checked that
$\Delta(j+1,j+3,j,m)=\Delta(j+1,j+3,j,j+2)+\Delta(j+1,j+3,j+2,m)$.
(For an illustration see Figure \ref{fig:schema}.)
Since $\Delta(j+1,j+3,j+2,m)=\Delta(j+2,m,j+1,j+3)$, it follows
from (\ref{vdv.c}) that both terms in the sum are non-positive. 

If $l>j+3$, then $l=j+3+2p$ with $p\ge 1$,
    since $l$ and $j$ are of different colours. It is easy to check that
    $\Delta(j+1,l,j,m)=\Delta(j+1,j+3,j,m)+\Delta(j+3,l,j,m)$, and
    $\Delta(j+1,j+3,j,m)\le 0$. If $l=j+3+2$, then
    $\Delta(j+3,l,j,m)\le 0$, as was shown above.
    Otherwise we represent $\Delta(j+3,l,j,m)$ as a sum of two terms,
    and so on. Eventually we end up with proving the inequality
    $\Delta(j+1,l,j,m)\le 0$.
\smallskip
    
\noindent
{\em Proof of Statement II:\/}
If $m=j+2$,  then it follows from (\ref{vdv.c}) that $\Delta(j+1,l,j,j+2)\le 0$. 
If $m>j+2$, then $l>j+3$, and it is easy to check that 
$\Delta(j+1,l,j,m)=\Delta(j+1,l,j,j+2)+\Delta(j+1,j+3,j+2,m)+\Delta(j+3,l,j+2,m)$. Again,
$\Delta(j+1,l,j,j+2)\le 0$ and $\Delta(j+1,j+3,j+2,m)\le 0$ due to
(\ref{vdv.c}). If $l-j-3>2$, we continue the process and eventually we end up
with proving the 
inequality $\Delta(j+1,l,j,m)\le 0$.
\hfill\qed
\medskip

Note that in the proof above only the following subset of the Van der Veen
inequalities is needed
\begin{eqnarray}
\label{eq:delta1}
\begin{blockarray}{cc}
c_{j+1,j}+c_{lm} -c_{jl}-c_{j+1,m}\le 0\ & 
 \mbox{for all}\ j,l,m:\ 1\le j<j+1<l,m;\\
&(j\equiv m)\ \mathrm{mod}\ 2 \mbox{ and } (j+1\equiv l)\ \mathrm{mod}\ 2.
\end{blockarray}
\end{eqnarray}
Note that all matrix entries involved in (\ref{eq:delta1}) are entries of
the submatrix $C[K_1,K_2]$.

We call a matrix $C$ that satisfies the conditions (\ref{eq:delta1}) a
\emph{relaxed} Van der Veen matrix.  The following corollary then follows
immediately.

\begin{corollary}
  The even-odd BTSP on an $n\times n$ relaxed Van der Veen matrix is
  pyramidally solvable.
\end{corollary}



For the sake of illustration, Figure~\ref{fig:pyr1} shows two
set of points in the Euclidean plane
along with their corresponding optimal BTSP tours.
The distance matrix in instance \emph{A} is a
Van der Veen matrix while the distance matrix in instance \emph{B} is
a relaxed Van der Veen matrix. In instance \emph{B} we chose the first
four points (as in instance A) and randomly generated the other points
to agree with conditions (\ref{eq:delta1}). In instance \emph{B},
17 out of 165  Van der Veen inequalities (\ref{vdv.c})
are violated; e.g., take the inequality that results for $i=6$, $j=9$ and
$m=11$.

\begin{figure}
\unitlength=1cm
\begin{center}
\begin{picture}(15.,5)
{
\begin{picture}(15.,5)
\put(1.2,0.7){\framebox[1\width]{
\includegraphics[scale=1]{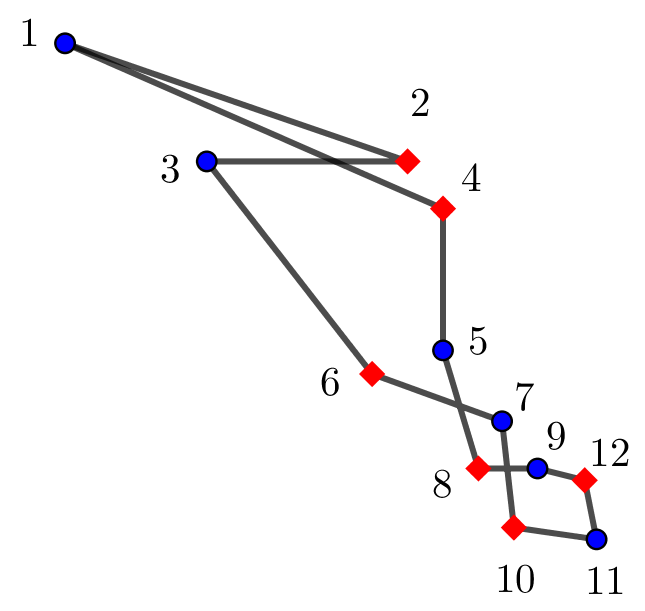}
}}
\put(8,0.7){\framebox[1\width]{
\includegraphics[scale=1]{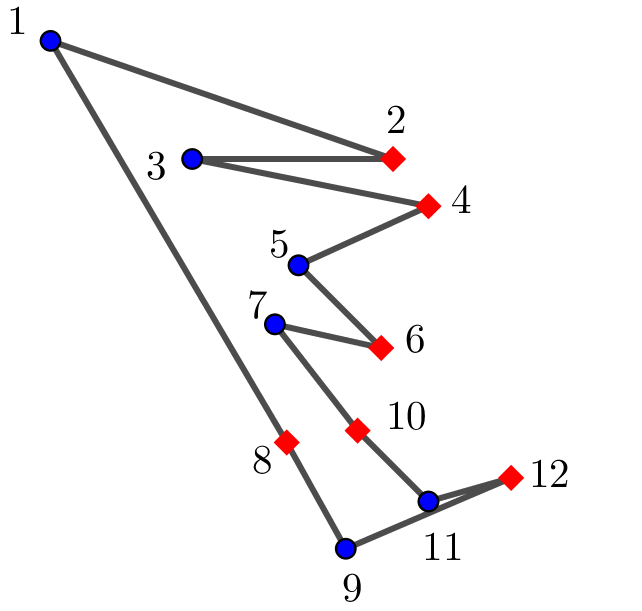}
}}
\put(3,0){Instance \emph{A}}
\put(10,0){Instance \emph{B}}
\end{picture} 
}
\end{picture}
\end{center}
\centerline
{
\begin{tabular}{|c|c|c|c|c|c|c|c|c|c|c|c|c|}
\hline 
Point number & 1 & 2 & 3 & 4 & 5 & 6 & 7 & 8 & 9 & 10 & 11 & 12 \\ 
\hline 
$x$-coord. in \emph{A} & 5 & 34 & 17 & 37 & 37 & 31 & 42 & 40 & 45 & 43 & 50& 49 \\ 
\hline 
$y$-coord. in \emph{A} & 45 &35 & 35 & 31 & 19 & 17 & 13 & 9 & 9 & 4 & 3 & 8 \\ 
\hline 
$x$-coord. in  \emph{B} & 5 & 35 & 17 & 37 & 26 & 33 & 24 & 25 & 30 & 31 & 37& 44 \\ 
\hline 
$y$-coord. in \emph{B}& 45 &35 & 35 & 31 & 26 & 19 & 21 & 11& 2 & 12 & 6 &8 \\ 
\hline 
\end{tabular}
} 
\caption{ Optimal  tours for instances of the even-odd BTSP with a Van der Veen matrix
  (Instance \emph{A}) and with a relaxed Van der Veen matrix (Instance \emph{B}).}   
\label{fig:pyr1}
\end{figure}

\section{Recognition of a subclass of permuted relaxed Van der Veen matrices}\label{sec:RecognitionP}

It is obvious that a matrix property which is defined by linear inequalities as
it is the case for Van der Veen matrices depends on the numbering of the rows
and of the columns. For the BTSP the partitioning of the set
of cities into the coloured sets also has to be taken into consideration.

We henceforth assume that the partitioning of the set of cities into the
subsets $K_1$ and $K_2$ is given, but we have the freedom of choosing the
numbering of the cities in each subset. More specifically, we consider symmetric
matrices that satisfy the system of linear inequalities (\ref{eq:delta1}) with
the partitioning $K_1:=\{1,3,\ldots,2k-1\}$ and $K_2:=\{2,4,\ldots,2k\}$.
We assume that the initial numbering of the cities in $K_1$ and $K_2$ was chosen
randomly, and the system (\ref{eq:delta1}) is not satisfied. We pose the
question whether it is possible to recognise \emph{the right} numbering of
the cities in $K_1$ and $K_2$.

To simplify the further notations, we define a new $k\times k$ asymmetric
matrix $A:=C[K_1,K_2]$ with $a_{ij}=c_{2i-1,2j}$. Using this notation,
the system (\ref{eq:delta1}) can be rewritten as
\begin{eqnarray}\label{eq:a1}
a_{11}+a_{lm}\le & a_{1m}+a_{l1} & l,m=2,3,\ldots,k,\\
a_{i,i-1}+a_{l,m}\le & a_{i,m}+a_{l,i-1} & l=i+1,\ldots,k; m=i,\ldots,k,\label{eq:a3}\\
a_{ii}+a_{lm}\le & a_{im}+a_{li}& l=i+1,\ldots,k; m=i+1,\ldots,k, \label{eq:a2}\\
&&i=2,\ldots,k-1.\nonumber
\end{eqnarray}

\begin{proposition}
Given a $k\times k$ matrix $A=(a_{ij})$, it can be decided in $O(k^4)$
time whether there exist permutations $\gamma$ and $\delta$ such that
the permuted matrix $(a_{\gamma(i)\delta(j)})$ satisfies conditions
\emph{(\ref{eq:a1})-(\ref{eq:a2})}. If the permutations
$\gamma$ and $\delta$ exist, they can be found in time $O(k^4)$.
\end{proposition}
\begin{proof} 
  The System (\ref{eq:a1})-(\ref{eq:a2}) is similar to the systems
  investigated in~\cite{BurDei} and~\cite{DW2014}. The proof below
  follows the logic of the approached used by~\cite{DW2014}. 

First try all $k$ indices as candidates for the first position
in $\gamma$. Let $\gamma(1)=1$. According to (\ref{eq:a1}), index $i$
can be placed at the first
position in $\delta$ if and only if 
\begin{eqnarray}
a_{1i}+a_{st}\le a_{si}+a_{1t} \mbox{ for all\ } s\neq 1, t\neq
i. \label{recognition}
\end{eqnarray}
If there is another candidate $j$ with the same property ($i\neq j$), then it
can be shown
that $ a_{1i}+a_{sj}=a_{si}+a_{1j}$, i.e., $ a_{sj}=a_{si}+d $
 for all $s$, where
$d=a_{1i}-a_{1j}$ is a constant for fixed $i$ and $j$. Since adding a
constant to a row or a column of matrix $A$ does not affect
the inequalities (\ref{eq:a1})-(\ref{eq:a2}), any of the indices $i$ or
$j$ can be placed at the first position in $\delta$.

The candidate $i$ can be chosen in $O(k^2)$
time. Note that the transformation $a'_{st}=a_{st}-a_{1t}$,
$s=1,\ldots,k$, $t=1,\ldots,k$, transforms
matrix $A$ into matrix $A'$ with zeros in the first row. 
The inequalities (\ref{recognition}) are equivalent to
$ a'_{st}\le a'_{si} \mbox{ for all\ } s,t$ and $i$.  
Clearly, index $i$ can be found in $O(k^2)$ time by looking through the
indices of the maximal entries in the rows of $A'$. 

An index for the second position in $\delta$ needs to be chosen by using the
same approach based on inequalities (\ref{eq:a3}). An index for the
second position in $\gamma$ needs to be chosen by using the inequalities
(\ref{eq:a2}). This approach is going to be repeated for all positions.

This results in $O(k^3)$ time needed for 
for each candidate at the position
$\gamma(1)$ and, therefore, overall a running time of
$O(k^4)$ as claimed.  \hfill\qed
\end{proof}
\smallskip

To illustrate the algorithm we consider the BTSP with a rectilinear distance matrix
(see Fig.~\ref{fig:pyrM1}) where the distances between points are calculated as $c_{ij}=|x_i-x_j|+|y_i-y_j|$. 
\begin{figure}
\unitlength=1cm
\begin{center}
\begin{picture}(15.,6)
{
\begin{picture}(15.,6)
\put(0.5,0.7){\framebox[1\width]{
\includegraphics[scale=0.99]{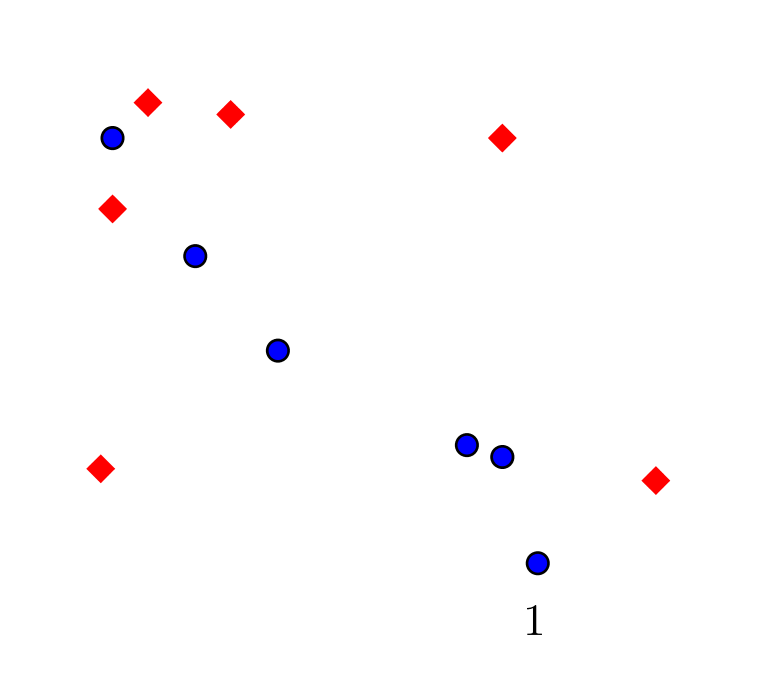}
}}
\put(7.5,0.7){\framebox[1\width]{
\includegraphics[scale=1]{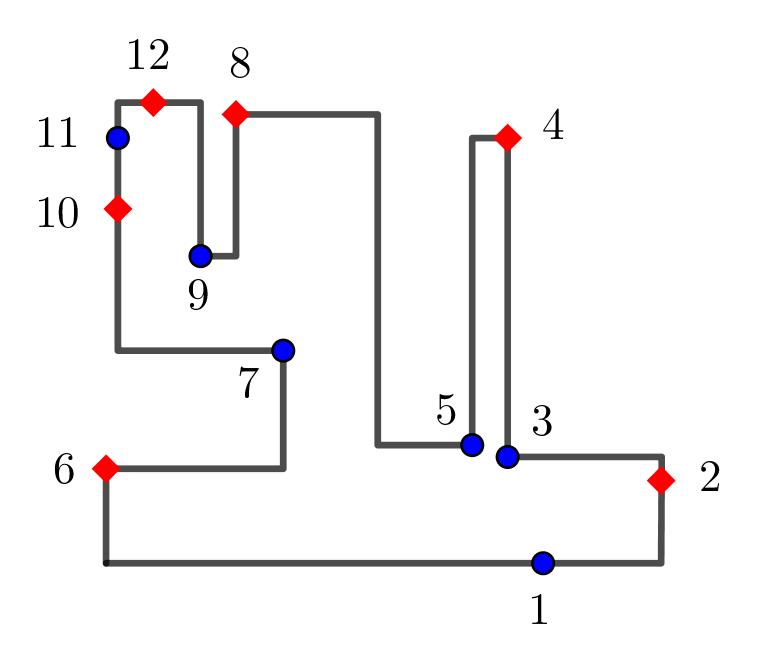}
}}

\put(3.5,0){(a)}
\put(10.5,0){(b)}
\end{picture} 
}
\end{picture}
\end{center}
\centerline
{
\begin{tabular}{|c|c|c|c|c|c|c|c|c|c|c|c|c|}
\hline 
Point number & 1 & 2 & 3 & 4 & 5 & 6 & 7 & 8 & 9 & 10 & 11 & 12 \\ 
\hline 
$x$-coord.  & 38 & 48 & 35 & 35 & 32 & 1 & 16 & 12 & 9 & 2 & 2& 5 \\ 
\hline 
$y$-coord.  & 8 &15 & 17 & 44 & 18 & 16 & 26 & 46 & 34 & 38 & 44 & 47 \\ 
\hline 
\end{tabular}
} 
\caption{ (a) - Set of points which satisfies (\ref{eq:a1})-(\ref{eq:a2}), if the numbering is chosen as shown in (b); (b)- Optimal pyramidal BTSP tour $\seq{1,2,3,4,5,8,9,12,11,10,7,6}$.}   
\label{fig:pyrM1}
\end{figure}

\[
A=
\begin{blockarray}{cccccc}
\begin{block}{(cccccc)}
  17 & 39 & 45 & 64 & 66 & 72\\
   15 & 27 & 35 & 52 & 54 & 60\\
 19 & 29 & 33 & 48 & 50 & 56\\
   43 & 37 & 25 & 24 & 26 & 32\\
 58 & 36 & 26 & 15 & 11 & 17\\
  75 & 33 & 29 & 12 & 6 & 6\\
\end{block}
\end{blockarray}\ \ \quad
A'=
\begin{blockarray}{cccccc}
\begin{block}{(cccccc)}
  0 & 0 & 0 & 0 & 0 & 0\\
  -2 & -12 & -10 & -12 & -12 & -12\\
   2 & -10 & -12 & -16 & -16 & -16\\
  26 & -2 & -20 & -40 & -40 & -40\\
   41 & -3 & -19 & -49 & -55 & -55\\
   58 & -6 & -16 & -52 & -60 & -66\\
\end{block}
\end{blockarray}
\]

The index of the maximal entries in $A'$ in rows $2,\ldots,6$ is $1$,
so $\delta(1)=1$, i.e., column 1 remains in $A$ at the same position.

Row 1 in $A$ is not relevant any more to further constructions, therefore we consider a $5\times 6$ submatrix of $A$ to choose a row to be placed at the second position in permutation $\gamma$. 
This submatrix and its transformation $A'$ are shown below.
\[
A_{5\times 6}=
\begin{blockarray}{cccccc}
\begin{block}{(cccccc)}
   15 & 27 & 35 & 52 & 54 & 60\\
 19 & 29 & 33 & 48 & 50 & 56\\
   43 & 37 & 25 & 24 & 26 & 32\\
 58 & 36 & 26 & 15 & 11 & 17\\
  75 & 33 & 29 & 12 & 6 & 6\\
\end{block}
\end{blockarray}\ \ \quad
A'_{5\times 6}=
\begin{blockarray}{cccccc}
\begin{block}{(cccccc)}
 0 & 12 & 20 & 37 &39 & 45\\
 0 & 10 & 14 & 29 & 31 & 37\\
 0 & -6 & -18 & -19 & -17 & -11\\
 0& -22& -32 & -43 & -47 & -41 \\
 0& -42&-46&-63&-69&-69\\
\end{block}
\end{blockarray}
\]

The index of the maximal entries in $A'$ in columns $2,\ldots,6$ corresponds
to row 2 in $A$ (row 1 in the $5\times 6$ submatrix), so $\gamma(2)=2$.

Proceeding in the same way we eventually obtain $\gamma=\delta=id_6$ where
$id_6$ denotes the identity permutation on $\{1,\ldots,6\}$ which
means that the initial numbering of the cities
as shown in Figure \ref{fig:pyrM1} yields a distance matrix that already
satisfies the conditions (\ref{eq:a1})-(\ref{eq:a2}), or equivalently,
conditions (\ref{eq:delta1}).

The optimal BTSP tour can be found by finding a shortest pyramidal tour,
which is the tour $\seq{1,2,3,4,5,8,9,12,11,10,7,6}$.

\section{Conclusion}\label{sec:conclusion}

In this paper we provided a new polynomially solvable case of the BTSP.
In previously published papers, an optimal solution can either be implicitly
specified based only on the knowledge that the distance matrix has a special
structure (\cite{DW2014, Halton,Mis}), or can be found in an exponential
neighbourhood in $O(n^6)$ time (\cite{Garcia}). In the new case discussed in
this paper an optimal solution belongs to the set of pyramidal tours,
and hence can be found in $O(n^2)$ time. If the rows and columns in
the distance matrix are permuted, the special structure of
the underlying distance matrix can be recognised in $O(n^4)$ time.

\paragraph{Acknowledgements.} This research has been supported by the
Austrian Science Fund (FWF): W1230. The authors thank Natalia Chakhlevitch  and Sigrid Knust for helpful discussions and comments on earlier version of this paper, in particular for a helpful reference to the crane loading problem.

\end{document}